\documentclass[a4paper,11pt,reqno]{article}
\usepackage{amssymb}
\usepackage[frenchb,english]{babel}
\usepackage{amscd}
\usepackage{amsmath,amsthm,amsfonts,amssymb,graphicx}
\usepackage[colorlinks,linkcolor=blue,anchorcolor=blue,citecolor=green]{hyperref}
\usepackage{mathrsfs}
\usepackage{fancyhdr}
\usepackage{multicol}
\usepackage{abstract}
\usepackage[latin1]{inputenc}
\usepackage{indentfirst}

\makeatother

\begin{document}
\theoremstyle{plain}
\newtheorem{Definition}{Definition}[section]
\newtheorem{Proposition}{Proposition}[section]
\newtheorem{Property}{Property}[section]
\newtheorem{Theorem}{Theorem}[section]
\newtheorem{Lemma}[Theorem]{\hspace{0em}\bf{Lemma}}
\newtheorem{Corollary}[Theorem]{Corollary}
\newtheorem{Remark}{Remark}[section]

\setlength{\oddsidemargin}{ 1cm}  
\setlength{\evensidemargin}{\oddsidemargin}
\setlength{\textwidth}{13.50cm}
\vspace{-.8cm}

\noindent  {\LARGE \begin{center}
The regular quantizations of certain holomorphic bundles
\end{center}}

\vskip 20pt

\noindent\text{Zhiming Feng  }\\
\noindent\small {School of Mathematical and Information Sciences, Leshan Normal University, Leshan, Sichuan 614000, P.R. China } \\
\noindent\text{Email: fengzm2008@163.com}

\vskip 20pt

\normalsize \noindent\textbf{Abstract}\quad {In this paper, we study the regular quantizations of K\"{a}hler manifolds by using the first two coefficients of Bergman function expansions. Firstly, we obtain sufficient and necessary conditions for certain Hermitian holomorphic vector bundles and their ball subbundles to be regular quantizations. Secondly,  we obtain that some projective bundles over the Fano manifolds $M$ admit regular quantizations if and only if $M$ are biholomorphically isomorphism to the complex projective spaces. Finally, we obtain the balanced metrics on certain Hermitian holomorphic vector bundles and their ball subbundles over the Riemann sphere.
\smallskip\\\
\smallskip\\\

\setlength{\oddsidemargin}{-.5cm}  
\setlength{\evensidemargin}{\oddsidemargin}
\pagenumbering{arabic}
\renewcommand{\theequation}
{\arabic{section}.\arabic{equation}}
\setcounter{section}{0}
 \setcounter{equation}{0}
\section{Introduction and main results}

Let $(L,h,\pi)$ be the quantum line bundle over a K\"{a}hler manifold $(M,g)$ of complex dimension $n$, namely, $(L, h)$ is a positive Hermitian holomorphic line bundle  such that $c_1 (L,h)=\omega$, where $\omega$ denotes the K\"{a}hler form associated to the K\"{a}hler  metric $g$,  $c_1(L,h)$ denotes the curvature of the Chern connection on the
Hermitian holomorphic bundle $(L,h)$, and $\pi:L\rightarrow M$ is the bundle projection.  The curvature $c_1(L,h)$ is given by
$$c_1(L,h) = -\frac{\sqrt{-1}}{2\pi}\partial\bar{\partial}\log h(\sigma(z), \sigma(z))$$
for a trivializing holomorphic section $\sigma : U \rightarrow L\backslash\{0\}$. The quantum line bundle $(L,h)$  is also called a geometric quantization of the K\"ahler manifold $(M, \omega)$.

For a given positive integer $m$,  let $(L^{m},h^{m})$ be the $m$-th tensor power of $L$ with $c_1(L^m,h^m)=m\omega$, and $H^0(M,L^m)$  be the space  consisting of global holomorphic sections of $L^m$. Define
\begin{equation}\label{e1.1}
  \mathcal{H}_m^2(M):=\left\{s\in H^0(M,L^m): \|s\|^2:=\int_M h^m(s(z),s(z))\frac{\omega^n}{n!}<+\infty \right\}.
\end{equation}
If \;$ \mathcal{H}_m^2(M)\neq \{0\}$,  the Bergman function on $M$  with respect to the metric $mg$ is defined by
\begin{equation}\label{e1.2}
 \epsilon_{mg}(z)=\sum_{j=1}^{\dim  \mathcal{H}_m^2(M)}h^m(s_j(z),s_j(z)),\;\; z\in M,
\end{equation}
where $\{s_j: 1\leq j\leq \dim \mathcal{H}_m^2(M)\}$ is an orthonormal basis for the  Hilbert space $ \mathcal{H}_m^2(M)$.

For compact manifolds by Catlin \cite{Cat} and Zelditch \cite{Zeld}, and for non-compact manifolds  by Ma-Marinescu \cite{MM07,MM08} and Engli\v{s} \cite{E1},
$\epsilon_{mg}(z)$ admits an asymptotic expansion as $m\rightarrow +\infty$
\begin{equation}\label{e1.3}
  \epsilon_{mg}(z)\sim\sum_{j=0}^{+\infty} a_j^{(g)}(z)m^{n-j},
\end{equation}
where expansion coefficients $a_j^{(g)}$ are polynomials of the curvature  and its covariant derivatives for the  metric $g$. For compact manifolds by Lu \cite{Lu} and for non-compact manifolds by Ma-Marinescu \cite{MM07} (Theorem 6.1.1) and Engli\v{s} \cite{E2}, the expansion coefficients $a_j^{(g)}$ have the same expression, in particular $a_0^{(g)}, a_1^{(g)}$ and $ a_2^{(g)}$ are given by
\begin{equation}\label{e1.4}
\left\{  \begin{array}{ll}
    a_0^{(g)} & =1, \\\\
    a_1^{(g)} & = \frac{1}{2}k_g, \\\\
    a_2^{(g)} & =\frac{1}{3}\triangle k_g+\frac{1}{24}|R_g|^2-\frac{1}{6}|\mathrm{Ric}_g|^2+\frac{1}{8}k_g^2.
  \end{array}\right.
\end{equation}
 Here $k_g$, $\triangle_g$, $R_g$ and $\mathrm{Ric}_g$  denote  the scalar curvature, the Laplace, the curvature tensor and the Ricci curvature associated to the
metric $g$, respectively.  For graph theoretic formulas of coefficients $a_j$, see Xu \cite{X}. For the more references of the Bergman function expansions, refer to Berezin \cite{Berezin},  Dai-Liu-Ma \cite{Dai-Liu-Ma}, Berman-Berndtsson-Sj\"ostrand \cite{B-B-S},  Ma-Marinescu \cite{MM12}, Hsiao \cite{HCY} and Hsiao-Marinescu \cite{HM2014}.

The K\"{a}hler metric $g$ on $M$ is balanced  if the Bergman function $\epsilon_{g}(z)$ $(z\in M)$ is a positive constant on $M$.   Balanced metric plays an important  role in the quantization  of a K\"{a}hler manifold, see Berezin \cite{Berezin}, Rawnsley \cite{Ra}, Cahen-Gutt-Rawnsley \cite{CGR-1,CGR-2}, Engli\v{s} \cite{E0}, Arezzo-Loi \cite{Arezzo-Loi},  Ma-Marinescu \cite{MM07} and Lui\'c \cite{Lui}. Balanced metrics play a central role when the polarized algebraic manifolds admit K\"{a}hler metrics of constant scalar curvature, see Donaldson \cite{Donaldson}. For the study of the balanced metrics,  see also Engli\v{s} \cite{E0,E3},  Loi-Mossa \cite{Loi-Mossa}, Loi-Zedda \cite{Loi-Zed, LZ}, Loi-Zedda-Zuddas \cite{Loi-Zed-Zud} and Arezzo-Loi-Zuddas \cite{Arezzo-Loi-Zuddas}.

If there is a positive integer $m_0$ such that Bergman functions $\epsilon_{mg}(z)$ are  positive  constants on $M$ for all $m\geq m_0$, $(L,h)$ is called a regular  quantization of the K\"ahler manifold $(M,\omega)$, for regular quantizations of compact complex manifolds, see Cahen-Gutt-Rawnsley \cite{CGR-1,CGR-2}. Cahen-Gutt-Rawnsley \cite{CGR-1} have shown that a geometric quantization $(L, h)$ of a homogeneous and simply connected compact K\"ahler manifold $(M,\omega)$ is regular. If $(M,\omega)$ admits a regular quantization, Cahen-Gutt-Rawnsley \cite{CGR-1} also have generalized Berezin's method \cite{Berezin} to the case of compact K\"ahler manifolds and to obtain a deformation quantization of the K\"ahler manifold $(M,\omega)$,  Loi in \cite{Loi1}  has proved that there exists an asymptotic expansion on $(M,\omega)$ for $\epsilon_{mg}(z)$ as \eqref{e1.3} and \eqref{e1.4}, and all coefficients $a_j^{(g)}$ are constants. For the study of the regular quantization,  also refer to  Arezzo-Loi \cite{Arezzo-Loi} and Loi \cite{Loi3}.

For the case of  non-compact manifolds, Loi-Mossa in \cite{Loi-Mossa} have shown that a bounded homogeneous domain admits a regular quantization. For the nonhomogeneous setting,  we give the existence of regular quantizations on some Hartogs domains in \cite{B-F-T} and \cite{FT2}.

In \cite{Feng}, the author studied constant scalar curvature K\"{a}hler metrics such that the second coefficients $a_2^{(g)}$ of the Bergman function expansions are constants on certain  Hartogs domains,  namely the part $(\mathrm{I})$ of Theorem \ref{apth:3.3} below with $\lambda=1$.

In this paper, we use the coefficients of Bergman function expansions to study the existence of regular quantizations of K\"{a}hler manifolds. Firstly, we use Theorem \ref{apth:3.3} to study regular quantizations of trivial Hermitian holomorphic vector bundles and their ball subbundles, and get Theorem \ref{Th:1.2}. Secondly, we study regular quantizations of certain Hermitian holomorphic vector bundles and their ball subbundles according to Theorem \ref{Th:1.2}, we get Theorem \ref{Th:1.3}. Thirdly, we study regular quantizations of compactification of certain Hermitian holomorphic vector bundles over compact complex manifolds, we obtain Theorem \ref{Th:1.4}. Finally, using Theorem \ref{Th:1.3} we get the balanced  metrics on certain Hermitian holomorphic vector bundles and their ball subbundles over the Riemann sphere $\mathbb{CP}^1$, namely Corollary \ref{Cor:1.4}. The main results of this article are described below.

Before describing Theorem \ref{Th:1.3}, we first define quantum line bundles over certain  holomorphic vector bundles. Let $(L_0^{*},h_{0}^{*},\pi_0)$ be the dual bundle of the quantum line bundle $(L_0,h_0,\pi_0)$ over a connected K\"{a}hler manifold $(M_0,\omega_0)$ of complex dimension $d$. For a given positive integer $r$, let $(L_0^{*\oplus r},h:=h_0^{*\oplus r})$ be the direct sum of $r$ copies of $(L_0^{*},h_{0}^{*})$, $B(L_0^{*\oplus r}):=\{v\in L_0^{*\oplus r}: h(v,v)<1\}$ be the ball subbundle of the Hermitian holomorphic vector bundle $(L_0^{*\oplus r},h)$,  $p: L_0^{*\oplus r}\rightarrow M_0$ and $\Pi:L:=p^{*}L_0\rightarrow L_0^{*\oplus r}$ be the bundle projections, where $L:=p^{*}L_0$ be pull-back of $L_0$ under the map $p$, that is
\begin{equation}\label{e1.5}
 \CD
  L=p^{*}L_0 @> >> L_0 \\
  @V \Pi VV @V \pi_0 VV  \\
  L_0^{*\oplus r} @>p>> M_0
\endCD,\;\;\;\;\;\;
\CD
 L=p^{*}L_0 @> >> L_0 \\
  @V \Pi VV @V \pi_0 VV  \\
  B(L_0^{*\oplus r}) @>p>> M_0
\endCD.
\end{equation}

Let $M:=L_0^{*\oplus r}$ or $B(L_0^{*\oplus r})$. Given a real smooth function $F$ with $F(0)=0$, let
 \begin{equation}\label{e1.8}
   h_F(\cdot,\cdot):=e^{-F(h(\Pi\cdot,\Pi\cdot))}\times (p^{*}h_0)(\cdot,\cdot)
 \end{equation}
be a Hermitian metric on the line bundle $L$ over $M$, where the Hermitian metric $p^{*}h_0$ on line bundle $L$ is pull-back of $h_0$ under the map $p$. If $\omega_F:=c_1(L,h_F)>0$, then $(L,h_F,\Pi)$ is the quantum line bundle over $(M,\omega_F)$. Let $g_F$ be a K\"{a}hler metric associated to the K\"{a}hler form $\omega_F=c_1(L,h_F)$ on $M$.

\begin{Theorem}\label{Th:1.3}{Under the assumptions above,  further suppose that

$(\mathrm{i})$  There exists a non-negative integer $m_0$ such that
\begin{equation*}
  \mathcal{H}_m^2(M_0):=\left\{s\in H^{0}(M_0,L_0^m):\|s\|^2:=\int_{M_0}h^m_0(s,s)\frac{\omega_{0}^{d}}{d!}<+\infty\right\}\neq \{0\}
\end{equation*}
for all $m>m_0$.

 $(\mathrm{ii})$ There exists a dense open contractible subset $U\subset M_0$ such that $M_0-U$ is an analytic subset of $M_0$, thus the restriction of $L_0$ to $U$ is the trivial holomorphic line bundle.

 $(\mathrm{iii})$ $(L,h_F,\Pi)$ is the quantum line bundle over $(M,\omega_F)$.

 $(\mathrm{iv})$ The fibre metric of $g_F$ is complete.

Then we have the following conclusions.

 $(\mathrm{I})$ There exists a positive integer $m_1$ such that the Bergman functions $\epsilon_{m g_F}$ for $(L^m,h^m_F)$ are constants  on $M=B(L_0^{*\oplus r})$  for all $m\geq m_1$ if and only if
\begin{equation}\label{e1.12}
  \left\{\begin{array}{l}
    F(\rho)   =-\frac{1}{A}\log(1-\rho),\;A>0, \\\\
     \epsilon_{m g_0} =m+r-(1+r)A,  \\\\
       \epsilon_{m g_F} = \prod_{j=1}^{1+r}(m-jA), \\\\
     m\geq m_1>\max \left\{m_0,(1+r)A\right\}
   \end{array}\right.
\end{equation}
for the case of $d=1$, and
\begin{equation}\label{e1.13}
  \left\{\begin{array}{l}
    F(\rho)   =-\log(1-\rho), \\\\
    \epsilon_{m g_0}= \prod_{j=1}^{d}(m-j), \\\\
      \epsilon_{m g_F} = \prod_{j=1}^{n}(m-j), \\\\
    m\geq m_1>\max\{m_0, n\}
   \end{array}\right.
\end{equation}
for the case of $d>1$, where $\epsilon_{m g_0}$ are Bergman functions for $(L_0^m,M_0,h^m_0)$, and $n=d+r$ denotes the dimension of $M$.

$(\mathrm{II})$ There exists a positive integer $m_1$ such that the Bergman functions $\epsilon_{m g_F}$ are constants  on $M=L_0^{*\oplus r}$ for all $m\geq m_1$ if and only if
\begin{equation}\label{e1.14}
  \left\{\begin{array}{l}
            d=1,\\\\
            F(\rho)=c\rho,c>0, \\\\
    \epsilon_{m g_0} =m+r, \\\\
       \epsilon_{m g_F}  = m^{1+r}, \\\\
     m\geq m_1>m_0.
   \end{array}\right.
\end{equation}
}\end{Theorem}

\begin{Theorem}\label{Th:1.4}{
Let $(L_0, h_0,\Pi_0)$ be the quantum line bundle over a connected compact K\"{a}hler manifold $(M_0,\omega_0)$ of complex dimension $d$ with the first Chern class $c_1(M_0)>0$ when $d>1$. Assume that there exists a dense open contractible subset $\Omega\subset M_0$ such that $M_0-\Omega$ is an analytic subset of $M_0$, thus the restriction of $L_0$ to $\Omega$ is the trivial holomorphic line bundle.  Suppose that
\begin{equation*}
  \mathcal{H}_m^2(M_0):=\left\{s\in H^{0}(M_0,L_0^m):\|s\|^2:=\int_{M_0}h^m_0(s,s)\frac{\omega_{0}^{d}}{d!}<+\infty\right\}\neq \{0\}
\end{equation*}
for all $m\in \mathbb{N}$.

Let $E$ be the direct sum of $r$ copies of $L_0$, i.e. $E = L_0^{\oplus r}$ with an associated hermitian metric  still denoted by $h_0$. Denote $\mathcal{O}_{M_0}$ as the structure sheaf of $M_0$. The projective bundle $M:=\mathbb{P}(E \oplus \mathcal{O}_{M_0})$ can be viewed as a compactification of $E$. The projection of the vector bundle $E$ to $M_0$ also denoted by $\Pi_0 : E \rightarrow M_0$.

Let
\begin{equation*}
  \omega(u)=(\Pi_0^{*}\omega_0)(u)+\frac{\sqrt{-1}}{2\pi}\partial\overline{\partial}F(h_0(u,u))\;\;(u\in E )
\end{equation*}
be a K\"{a}hler form on $E$ with $F(0)=0$ such that

$(\mathrm{i})$ $\omega$ can be extended across $M-E$ (the extension of $\omega$ on $M$ is still expressed as $\omega$);

$(\mathrm{ii})$ $(M,\omega)$ admits a geometric quantization $(L, h)$.

$(\mathrm{iii})$ Bergman functions $\epsilon_{m g}$ for $(L^m,h^m)$ are constants on $M$  for all $m\geq 1$, where $g$ is a K\"{a}hler  metric  associated with  the K\"{a}hler form $\omega$.

Then $M_0$ and $M$ are biholomorphically isomorphism to the complex projective spaces $\mathbb{CP}^d$ and $\mathbb{CP}^{d+r}$, respectively, and $F(\rho)=\log(1+c\rho)$ with $c>0$.

}\end{Theorem}

Denote $\mathcal{O}_{\mathbb{CP}^d}(1)$ as the hyperplane line bundle of the $d$-dimensional complex projective space $\mathbb{CP}^d$, $\omega_{FS}$ as the standard Fubini-Study form on $\mathbb{CP}^d$. Let $(\mathcal{O}_{\mathbb{CP}^d}(1), h_{FS})$ be the quantum line bundle over the complex projective space $(\mathbb{CP}^d,\omega_{FS})$.
For a given $k\in \mathbb{N}_{+}$, let  $L=\mathcal{O}_{\mathbb{CP}^d}(k)$ and $h=h^k_{FS}$ be  the $k$-th tensor power of $\mathcal{O}_{\mathbb{CP}^d}(1)$ and $h_{FS}$, respectively.
Write $(\mathcal{O}_{\mathbb{CP}^n}(-k),h^{-k}_{FS})$ as the dual bundle of $(\mathcal{O}_{\mathbb{CP}^n}(k),h^k_{FS})$.
Let $(L_0,M_0,h_0)=(\mathcal{O}_{\mathbb{CP}^1}(k),\mathbb{CP}^1,h_{FS}^k)$ in Theorem \ref{Th:1.3}, using $\epsilon_{m g_0}=m+\frac{1}{k}$, we obtain the following Corollary \ref{Cor:1.4}.

\begin{Corollary}\label{Cor:1.4}{For given positive integers $k, r\in\mathbb{N}_{+}$, let
\begin{equation*}
 \CD
 L:= p^{*}\mathcal{O}_{\mathbb{CP}^1}(k) @> >> L_0:=\mathcal{O}_{\mathbb{CP}^1}(k) \\
  @V \Pi VV @V \pi_0 VV  \\
 \oplus_{i=1}^r\mathcal{O}_{\mathbb{CP}^1}(-k) @>p>> M_0:=\mathbb{CP}^1
\endCD,
\end{equation*}
where
 $$p: \oplus_{i=1}^r\mathcal{O}_{\mathbb{CP}^1}(-k)\rightarrow \mathbb{CP}^1$$
 and
 $$\Pi:p^{*}\mathcal{O}_{\mathbb{CP}^1}(k)\rightarrow \oplus_{i=1}^r\mathcal{O}_{\mathbb{CP}^1}(-k)$$
 are the bundle projections.

$(\mathrm{I})$ For $kr>1$, set
$$h=\oplus_{i=1}^rh_{FS}^{-k},\;\lambda=\frac{k(r+1)}{kr-1}, $$
$$M=\left\{v\in \oplus_{i=1}^r\mathcal{O}_{\mathbb{CP}^1}(-k): h(v,v)<1\right\} $$
and
$$h_F(u,u)=(1-h(\Pi(u),\Pi(u)))^{\lambda}\times (p^{*}h_{FS}^k)(u,u),\;u\in L|_{M}.$$
Let $g_F$ be the K\"{a}hler metric associated to the K\"{a}hler form $\omega_F:=c_1(L,h_F)$ on $M$, then $(L|_{M},h_F)$ is a quantum line bundle over $(M,g_F)$, and metrics $mg_F$ are balanced metrics on $M$ for all $m\geq r$.

$(\mathrm{II})$ For $c>0$, $k=r=1$, $M=\mathcal{O}_{\mathbb{CP}^1}(-1)$, let $g_F$ be the K\"{a}hler metric associated to the K\"{a}hler form $\omega_F:=c_1(L,h_F)$ on $M$, where
 $$ h_F(u,u)=e^{-c\times h_{FS}^{-1}(\Pi(u),\Pi(u))}\times (p^{*}h_{FS})(u,u),\;u\in L.$$
Then $(L,h_F)$ is a quantum line bundle over $(M,g_F)$, and metrics $mg_F$ are balanced metrics on $M$ for all $m\geq 1$.
}\end{Corollary}

\begin{Remark}
Recently, Aghedu-Loi \cite{Aghedu-Loi} has get the same conclusion as the part $(\mathrm{II})$ of Corollary \ref{Cor:1.4}.
\end{Remark}

The organization of this paper is as follows. In  Section 2, we study regular quantizations of trivial Hermitian holomorphic vector bundles and their ball subbundles.   In Section 3, Section 4 and Section 5,  we give proofs of Theorem \ref{Th:1.3},  Theorem \ref{Th:1.4} and Corollary \ref{Cor:1.4}, respectively.

 \setcounter{equation}{0}
\section{Regular quantizations of trivial vector bundles and their ball subbundles}

  From Lemma 2.4 and Lemma 2.5 of \cite{Feng}, we have the following lemma.

\begin{Lemma}\label{Le:3.1}{Let $g_{\phi}$ be a K\"{a}hler  metric  on a domain $\Omega\subset \mathbb{C}^d$ (that's the connected open set of $\mathbb{C}^d$) associated with the K\"{a}hler form $\omega_{\phi}=\frac{\sqrt{-1}}{2\pi}\partial\overline{\partial}\phi$, where $\phi$ is globally defined the K\"{a}hler potential on the domain $\Omega$.
For $0\neq \lambda\in \mathbb{R}$, set
$$ M:=\left\{(z,w)\in \Omega\times\mathbb{C}^{d_0}: \|w\|^2<e^{-\lambda\phi(z)}\right\}\;\text{or}\;\;\Omega\times\mathbb{C}^{d_0}.$$

Let $g_F$ be a K\"{a}hler  metric  on the domain $M$  associated with the K\"{a}hler form $\omega_F=\frac{\sqrt{-1}}{2\pi}\partial\overline{\partial}\Phi_F$, here
$$\Phi_F(z,w):=\phi(z)+F(\lambda\phi(z)+\log\|w\|^2).$$

Let $k_{g}$, $\Delta_{g}$, $\mathrm{Ric}_{g}$ and $R_{g}$ be the scalar curvature, the Laplace, the Ricci curvature and the curvature tensor with respect to the metric $g=g_{\phi}$ or $g_F$, respectively. Put
\begin{equation*}
 a_{1}^{(g_{\phi})}=\frac{1}{2}k_{g_{\phi}},\;a_{2}^{(g_{\phi})}=\frac{1}{3}\triangle_{g_{\phi}}k_{g_{\phi}}+\frac{1}{24}|R_{g_{\phi}}|^2-\frac{1}{6}|\mathrm{Ric}_{g_{\phi}}|^2+\frac{1}{8}k_{g_{\phi}}^2
\end{equation*}
and
\begin{equation*}
  a_1^{(g_F)}=\frac{1}{2}k_{g_F},\;a_2^{(g_F)}=\frac{1}{3}\triangle_{g_F} k_{g_F}+\frac{1}{24}|R_{g_F}|^2-\frac{1}{6}|\mathrm{Ric}_{g_F}|^2+\frac{1}{8}k_{g_F}^2.
\end{equation*}

Let
$$t=\lambda\phi(z)+\log \|w\|^2,\;x=F'(t),\;\varphi(x)=F''(t)\;,$$
\begin{equation*}
  \sigma=\frac{\left((1+\lambda x)^{d}x^{d_0-1}\varphi\right)'}{(1+\lambda x)^{d}x^{d_0-1}}
\end{equation*}
and
\begin{equation*}
 \chi=\frac{d_0(d_0-1)}{x}-\frac{\left((1+\lambda x)^{d}x^{d_0-1}\varphi\right)''}{(1+\lambda x)^{d}x^{d_0-1}}.
\end{equation*}

We have the following conclusions.

$(\mathrm{i})$
\begin{equation}\label{e5.54}
  k_{g_F}=\frac{1}{1+\lambda x}k_{g_{\phi}}+\frac{d_0(d_0-1)}{x}-\frac{\left((1+\lambda x)^{d}x^{d_0-1}\varphi\right)''}{(1+\lambda x)^{d}x^{d_0-1}},
\end{equation}
\begin{eqnarray}
\label{e5.55} |\mathrm{Ric}_{g_F}|^2  &=& \frac{|\mathrm{Ric}_{g_{\phi}}|^2-2\lambda\sigma k_{g_{\phi}}+d\lambda^2\sigma^2}{(1+\lambda x)^2}
    +(d_0-1)\left(\frac{\sigma-d_0}{x}\right)^2+(\sigma')^2,
\end{eqnarray}
\begin{eqnarray}
\nonumber \triangle_{g_F} k_{g_F}  &=&\frac{1}{(1+\lambda x)^2}(\triangle_{g_{\phi}}k_{g_{\phi}})-\frac{\left(\lambda(1+\lambda x)^{d-2}x^{d_0-1}\varphi\right)'}{(1+\lambda x)^{d}x^{d_0-1}}k_{g_{\phi}} \\
\label{e5.56}   & &+\frac{\left((1+\lambda x)^{d}x^{d_0-1}\varphi\chi'\right)'}{(1+\lambda x)^{d}x^{d_0-1}}
\end{eqnarray}
and
\begin{eqnarray}
\nonumber   & &  |R_{g_F}|^2\\
\nonumber   &=&\frac{1}{(1+\lambda x)^2}|R_{g_{\phi}}|^2-\frac{4\lambda^2\varphi}{(1+\lambda x)^3}k_{g_{\phi}}+\frac{2d(d+1)\lambda^4\varphi^2}{(1+\lambda x)^4}
+4d\lambda^2\left(\left(\frac{\varphi}{1+\lambda x}\right)'\right)^2\\
\label{e5.57}   & &+\left(\varphi''\right)^2+(d_0-1)\left\{4d \lambda^2 \left(\frac{\varphi}{x(1+\lambda x)}\right)^2+4\left(\left(\frac{\varphi}{x}\right)'\right)^2+2d_0\left(\frac{\varphi-x}{x^2}\right)^2\right\}.
\end{eqnarray}

$(\mathrm{ii})$

\begin{equation}\label{e5.48}
   a_1^{(g_F)}=\frac{a_{1}^{(g_{\phi})}}{1+\lambda x}+\frac{d_0(d_0-1)}{2x}-\frac{\left((1+\lambda x)^{d}x^{d_0-1}\varphi\right)''}{2(1+\lambda x)^{d}x^{d_0-1}}
\end{equation}
and
\begin{eqnarray}
\nonumber   a_2^{(g_F)}  &=& \frac{a_{2}^{(g_{\phi})}}{(1+\lambda x)^2}+\left\{\frac{\chi}{2(1+\lambda x)}+\frac{\lambda ^2\varphi}{(1+\lambda x)^3}\right\}a_{1}^{(g_{\phi})} \\
\nonumber   & & +\frac{1}{24}\left\{8(\varphi\chi')'+8\left(\frac{d\lambda}{1+\lambda x}+\frac{d_0-1}{x}\right)\varphi\chi'+3\chi^2-4(\sigma')^2+(\varphi'')^2\right.\\
\nonumber  & &\left.
+4d\lambda^2\left(\left(\frac{\varphi}{1+\lambda x}\right)'\right)^2
-\frac{4d\lambda^2}{(1+\lambda x)^2}\sigma^2+\frac{2d(d+1)\lambda^4}{(1+\lambda x)^4}\varphi^2\right\}\\
\label{e5.49}   & &+\frac{d_0-1}{6}\left\{\frac{d\lambda^2\varphi^2}{x^2(1+\lambda x)^2}+\left(\left(\frac{\varphi}{x}\right)'\right)^2+\frac{d_0}{2}\frac{(\varphi-x)^2}{x^4}
-\frac{(\sigma-d_0)^2}{x^2}\right\}.
\end{eqnarray}

$(\mathrm{iii})$ If $\varphi(x)=x+\lambda x^2$, then

\begin{equation}\label{Le3.1.1}
   a_1^{(g_F)}=-\frac{n(n+1)\lambda}{2}+\frac{1}{1+\lambda x}\left(\frac{d(d+1)\lambda}{2}+a_{1}^{(g_{\phi})}\right).
\end{equation}

\begin{eqnarray}
\nonumber   a_2^{(g_F)} &=& \frac{(n-1)n(n+1)(3n+2)\lambda^2}{24}+\frac{1}{1+\lambda x}\frac{(n-1)(n+2)\lambda}{2}\left(\frac{d(d+1)\lambda}{2}+a_{1}^{(g_{\phi})}\right)\\
\label{Le3.1.2} &&+\frac{1}{(1+\lambda x)^2}\left(a_{2}^{(g_{\phi})}+\frac{(d-1)(d+2)\lambda}{2}a_{1}^{(g_{\phi})}+\frac{(d-1)d(d+1)(3d+10)\lambda^2}{24}\right).
\end{eqnarray}
So, both $ a_1^{(g_F)}$ and $ a_2^{(g_F)}$ are constants if and only if
\begin{equation}\label{Le3.1.3}
  \left\{\begin{array}{l}
          a_{1}^{(g_{\phi})}=-\frac{d(d+1)\lambda}{2}, \\\\
          a_{2}^{(g_{\phi})}=\frac{(d-1)d(d+1)(3d+2)\lambda^2}{24}.
         \end{array}
  \right.
\end{equation}
}\end{Lemma}

\begin{Remark}
For other proofs of the formula \eqref{e5.54}, see  \cite{Hwang-Singer} and \cite{Fu-Yau-Zhou}.
\end{Remark}

\begin{proof}[Proof]

We only give a proof for the part $(\mathrm{i})$, because the part $(\mathrm{ii})$ and the part  $(\mathrm{iii})$  are derived directly from the part $(\mathrm{i})$.

 For the case of $\lambda>0$, let
\begin{equation*}
  \widetilde{\phi}=\lambda\phi,\;\widetilde{F}=\lambda F,\;\omega_{\widetilde{\phi}}=\lambda\omega_{\phi},\;\omega_{\widetilde{F}}=\lambda\omega_F,\;\widetilde{x}=\widetilde{F}',\;
  \widetilde{\varphi}=\widetilde{F}'',
\end{equation*}
\begin{equation*}
  \widetilde{\sigma}=\frac{\left((1+\widetilde{x})^{d}\widetilde{x}^{d_0-1}\widetilde{\varphi}\right)'}{(1+\widetilde{x})^{d}\widetilde{x}^{d_0-1}},\;
  \widetilde{\chi}=\frac{d_0(d_0-1)}{\widetilde{x}}-\frac{\left((1+\widetilde{x})^{d}\widetilde{x}^{d_0-1}\widetilde{\varphi}\right)''}{(1+\widetilde{x})^{d}\widetilde{x}^{d_0-1}}.
\end{equation*}
Then
\begin{equation}\label{Le3.1.8}
  k_{g_F}=\lambda k_{g_{\widetilde{F}}},\;|\mathrm{Ric}_{g_F}|^2=\lambda^2|\mathrm{Ric}_{g_{\widetilde{F}}}|^2,\;\triangle_{g_F} k_{g_F}=\lambda^2\triangle_{g_{\widetilde{F}}} k_{g_{\widetilde{F}}},\;|R_{g_F}|^2=\lambda^2|R_{g_{\widetilde{F}}}|^2,
\end{equation}
\begin{equation}\label{Le3.1.9}
  k_{g_{\phi}}=\lambda k_{g_{\widetilde{\phi}}},\;|\mathrm{Ric}_{g_{\phi}}|^2=\lambda^2|\mathrm{Ric}_{g_{\widetilde{\phi}}}|^2,\;\triangle_{g_{\phi}} k_{g_{\phi}}=\lambda^2\triangle_{g_{\widetilde{\phi}}} k_{g_{\widetilde{\phi}}},\;|R_{g_{\phi}}|^2=\lambda^2|R_{g_{\widetilde{\phi}}}|^2
\end{equation}
and
\begin{equation}\label{Le3.1.10}
  \widetilde{x}=\lambda x,\;\widetilde{\varphi}(\widetilde{x})=\lambda\varphi(x),\;\widetilde{\sigma}(\widetilde{x})=\sigma(x),\;\widetilde{\chi}(\widetilde{x})=\frac{1}{\lambda}\chi(x),\;
  \widetilde{\sigma}'(\widetilde{x})=\frac{1}{\lambda}\sigma'(x),\;\widetilde{\chi}'(\widetilde{x})=\frac{1}{\lambda^2}\chi'(x).
\end{equation}
where  $g_{\widetilde{\phi}}$ and $g_{\widetilde{F}}$ are K\"{a}hler metrics associated with the K\"{a}hler form $\omega_{\widetilde{\phi}}$ and $\omega_{\widetilde{F}}$, respectively.

By Lemma 2.4 and Lemma 2.5 of \cite{Feng}, we have
\begin{equation}\label{Le3.1.4}
  k_{g_{\widetilde{F}}}=\frac{1}{1+\widetilde{x}}k_{g_{\widetilde{\phi}}}+\frac{d_0(d_0-1)}{\widetilde{x}}-\frac{\left((1+\widetilde{x})^d\widetilde{x}^{d_0-1}\widetilde{\varphi}\right)''}{(1+\widetilde{x})^d\widetilde{x}^{d_0-1}},
\end{equation}
\begin{eqnarray}
\nonumber |\mathrm{Ric}_{g_{\widetilde{F}}}|^2  &=& \frac{1}{(1+\widetilde{x})^2}|\mathrm{Ric}_{g_{\widetilde{\phi}}}|^2-\frac{2\widetilde{\sigma}}{(1+\widetilde{x})^2} k_{g_{\widetilde{\phi}}}+(\widetilde{\sigma}')^2 \\
\label{Le3.1.5}   & & +d\left(\frac{\widetilde{\sigma}}{1+\widetilde{x}}\right)^2+(d_0-1)\left(\frac{\widetilde{\sigma}-d_0}{\widetilde{x}}\right)^2,
\end{eqnarray}
\begin{eqnarray}
\nonumber  & & \triangle_{g_{\widetilde{F}}} k_{g_{\widetilde{F}}}\\
\label{Le3.1.6}   &=& \frac{1}{(1+\widetilde{x})^2}(\triangle_{g_{\widetilde{\phi}}}k_{g_{\widetilde{\phi}}})- \frac{\left(\widetilde{\varphi}(1+\widetilde{x})^{d-2}\widetilde{x}^{d_0-1}\right)'}{(1+\widetilde{x})^d\widetilde{x}^{d_0-1}}k_{g_{\widetilde{\phi}}}+\frac{\left(\widetilde{\varphi}\widetilde{\chi}'(1+\widetilde{x})^d\widetilde{x}^{d_0-1}\right)'}{(1+\widetilde{x})^d\widetilde{x}^{d_0-1}}
\end{eqnarray}
and
\begin{eqnarray}
\nonumber    & &|R_{g_{\widetilde{F}}}|^2 \\
\nonumber     &=& \frac{1}{(1+\widetilde{x})^2}|R_{g_{\widetilde{\phi}}}|^2-\frac{4\widetilde{\varphi}}{(1+\widetilde{x})^3}k_{g_{\widetilde{\phi}}}+2d(d+1)\frac{\widetilde{\varphi}^2}{(1+\widetilde{x})^4}+4d\left(\left(\frac{\widetilde{\varphi}}{1+\widetilde{x}}\right)'\right)^2 \\
\label{Le3.1.7}   & & +\left(\widetilde{\varphi}''\right)^2+(d_0-1)\left\{4d \left(\frac{\widetilde{\varphi}}{\widetilde{x}(1+\widetilde{x})}\right)^2+4\left(\left(\frac{\widetilde{\varphi}}{\widetilde{x}}\right)'\right)^2+2d_0\left(\frac{\widetilde{\varphi}-\widetilde{x}}{\widetilde{x}^2}\right)^2\right\}.
\end{eqnarray}
The above \eqref{Le3.1.4}, \eqref{Le3.1.5}, \eqref{Le3.1.6} and \eqref{Le3.1.7} combine with  \eqref{Le3.1.8}, \eqref{Le3.1.9} and \eqref{Le3.1.10}, we get \eqref{e5.54}, \eqref{e5.55}, \eqref{e5.56} and \eqref{e5.57}.

For the case of $\lambda<0$, since $ k_{g_F}$, $|\mathrm{Ric}_{g_F}|^2$, $\triangle_{g_F} k_{g_F}$ and $|R_{g_F}|^2$ are rational functions in $\lambda$ under given any $(z,x)\in M$, it follows that  \eqref{e5.54}, \eqref{e5.55}, \eqref{e5.56} and \eqref{e5.57} still hold for $\lambda<0$.

\end{proof}

Using Lemma \ref{Le:3.1}, We get sufficient and  necessary conditions for both the first two coefficients of the Bergman function expansion for $(M,g_F)$ to be constants.

\begin{Theorem}\label{apth:3.3}{Under assumptions of Lemma \ref{Le:3.1}, let $F(-\infty)=0$ and $n=d+d_0$.

 $(\mathrm{I})$ For
$$ M=\left\{(z,w)\in \Omega\times\mathbb{C}^{d_0}: \|w\|^2<e^{-\lambda\phi(z)}\right\},$$
if the fibre metric of $g_F$ is complete, then both $ a_1^{(g_F)}$ and $ a_2^{(g_F)}$ are constants on $M$ if and only if

 $(\mathrm{i})$ For $d=1$,
\begin{equation}\label{e5.1}
 \left\{\begin{array}{cl}
          F(t) &= -\frac{1}{A}\log(1-e^t), \;A>0,\\\\
          a_{1}^{(g_{\phi})} & =d_0\lambda-n A, \;\lambda>0,\\\\
          a_{2}^{(g_{\phi})} &= 0.
        \end{array}
 \right.
\end{equation}

 $(\mathrm{ii})$ For $d>1$,
 \begin{equation}\label{e5.2}
 \left\{\begin{array}{cl}
          F(t) &= -\frac{1}{\lambda}\log(1-e^t),\lambda>0, \\\\
          a_{1}^{(g_{\phi})} & =-\frac{1}{2}d(d+1)\lambda, \\\\
          a_{2}^{(g_{\phi})} &= \frac{1}{24}(d-1)d(d+1)(3d+2)\lambda^2.
        \end{array}
 \right.
\end{equation}

 $(\mathrm{II})$ For $ M=\Omega\times\mathbb{C}^{d_0}$, if the fibre metric of $g_F$ is complete, then both $ a_1^{(g_F)}$ and $ a_2^{(g_F)}$ are constants on $M$ if and only if
\begin{equation}\label{e5.3}
 \left\{\begin{array}{cl}
            d&= 1,\\\\
          F(t) &= ce^t,\;c>0, \\\\
          a_{1}^{(g_{\phi})} & =d_0\lambda,\;\lambda>0, \\\\
          a_{2}^{(g_{\phi})} &= 0.
        \end{array}
 \right.
\end{equation}

$(\mathrm{III})$ For $ M=\Omega\times\mathbb{C}^{d_0}$, if the fibre metric of $g_F$ is incomplete, then both $ a_1^{(g_F)}$ and $ a_2^{(g_F)}$ are constants on $M$ if and only if

$(\mathrm{i})$ For $d=1$,
\begin{equation}\label{e5.4}
 \left\{\begin{array}{cl}
          F(t) &= -\frac{1}{A}\log(1+ce^t),A<0,\lambda\geq A,c>0, \\\\
          a_{1}^{(g_{\phi})} & =d_0\lambda-nA, \\\\
          a_{2}^{(g_{\phi})} &= 0.
        \end{array}
 \right.
\end{equation}

 $(\mathrm{ii})$ For $d>1$,
 \begin{equation}\label{e5.5}
 \left\{\begin{array}{cl}
          F(t) &= -\frac{1}{\lambda}\log(1+ce^t),\lambda<0,c>0, \\\\
          a_{1}^{(g_{\phi})} & =-\frac{1}{2}d(d+1)\lambda, \\\\
          a_{2}^{(g_{\phi})} &= \frac{1}{24}(d-1)d(d+1)(3d+2)\lambda^2.
        \end{array}
 \right.
\end{equation}

}\end{Theorem}

\begin{proof}[Proof]

Our first goal is to show that $\varphi$ is a polynomial if both $ a_1^{(g_F)}$ and $ a_2^{(g_F)}$ are constants.

By \eqref{e5.48} and \eqref{e5.49}, if both $ a_1^{(g_F)}$ and $ a_2^{(g_F)}$ are constants, then $a_{1}^{(g_{\phi})}$ and $a_{2}^{(g_{\phi})}$ are constants, thus $|R_{g_{\phi}}|^2-4|\mathrm{Ric}_{g_{\phi}}|^2$ and $|R_{g_F}|^2-4|\mathrm{Ric}_{g_F}|^2$ also are constants. From \eqref{e5.48}, $\varphi(x)$ can be written as
\begin{equation}\label{ne3.1}
  \varphi(x)=\sum_{j=0}^2A_jx^j+\sum_{j=1}^{d_0-1}\frac{B_j}{x^j}+\sum_{j=1}^{d}\frac{C_j}{(1+\lambda x)^j}.
\end{equation}

Substituting \eqref{ne3.1} into \eqref{e5.57} and \eqref{e5.55}, we have
\begin{eqnarray*}
   & & |R_{g_F}|^2-4|\mathrm{Ric}_{g_F}|^2 \\
   &=&\frac{p(d_0-1,d_0-1)B_{d_0-1}^2}{x^{2(d_0-1)+4}}+\frac{p(d,d)\lambda^4C_{d}^2}{(1+\lambda x)^{2d+4}}+\cdots,
\end{eqnarray*}
where
\begin{equation*}
p(d,k)=2d(d+1)+4d(k+1)^2+k^2(k+1)^2-4(d+(k+1)^2)(d-k)^2.
\end{equation*}

For $d_0>1$, from $p(d_0-1,d_0-1)>0$ and $p(d,d)>0$, we have $B_{d_0-1}=C_{d}=0$, so
\begin{eqnarray*}
   & & |R_{g_F}|^2-4|\mathrm{Ric}_{g_F}|^2 \\
   &=&\frac{p(d_0-1,d_0-2)B_{d_0-2}^2}{x^{2(d_0-2)+4}}+\frac{p(d,d-1)\lambda^4C_{d-1}^2}{(1+\lambda x)^{2(d-1)+4}}+\cdots,
\end{eqnarray*}
which follows that $B_{d_0-2}=0$ for $d_0>2$ and $C_{d-1}=0$ for $d>1$. Thus from \eqref{ne3.1}, $\varphi(x)$ can be written as
\begin{equation}\label{ne3.2}
  \varphi(x)=\sum_{j=0}^2A_jx^j+\sum_{j=1}^{d_0-3}\frac{B_j}{x^j}+\sum_{j=1}^{d-2}\frac{C_j}{(1+\lambda x)^j}.
\end{equation}

Let
\begin{equation*}
  S_0:=\{j:B_j\neq 0,1\leq j\leq d_0-3\},\; S_1:=\{j:C_j\neq 0,1\leq j\leq d-2\}.
\end{equation*}
If
\begin{equation*}
  S_0\neq\emptyset,\text{or}\;S_1\neq\emptyset,
\end{equation*}
let
\begin{equation*}
  k_j:=\max S_j,\;j=0,1.
\end{equation*}

Substituting \eqref{ne3.2} into \eqref{e5.48}, we obtain
\begin{equation*}
   a_1^{(g_F)}=-\frac{q(d_0-1,k_0)B_{k_0}}{2x^{k_0+2}}+\cdots
\end{equation*}
or
\begin{equation*}
   a_1^{(g_F)}=-\frac{q(d,k_1)\lambda^2C_{k_1}}{2(1+\lambda x)^{k_1+2}}+\cdots
\end{equation*}
for $k_0>0$ or $k_1>0$, respectively. Here
\begin{equation*}
  q(d,k)=(d-k)(d-k-1).
\end{equation*}
Then $B_{k_0}=0$ or $C_{k_1}=0$, this conflicts with the definitions of $k_j$, $j=0,1$. Namely, $B_j=0$ for $1\leq j\leq d_0-3$ and $C_j=0$ for $1\leq j\leq d-2$.

Now prove the rest of Theorem \ref{apth:3.3}.

Since the metric $g_F$ is defined at $t =-\infty$, it follows that $\varphi(0) = 0$ and $\varphi'(0)=1$. Hence $\varphi(x)=x+Ax^2$.

Substituting $\varphi(x)=x+Ax^2$ into \eqref{e5.48}, we have
\begin{equation}\label{ee2.17}
 2 a_1^{(g_F)}=-A(d+d_0+1)(d+d_0)+\frac{2a_{1}^{(g_{\phi})}+d(2Ad+2Ad_0-d\lambda-2d_0\lambda+\lambda)}{1+\lambda x}-\frac{d(d-1)(A-\lambda)}{(1+\lambda x)^2},
\end{equation}
it follows that
\begin{equation}\label{ee2.18}
  \left\{\begin{array}{l}
             d=1,\\\\
             \varphi(x)=x+Ax^2,\\\\
           a_{1}^{(g_{\phi})}=d_0\lambda-nA, \\\\
             a_1^{(g_F)}=-\frac{1}{2}n(n+1)A
                     \end{array}
  \right.
\end{equation}
and
\begin{equation}\label{ee2.19}
  \left\{\begin{array}{l}
             d>1,\\\\
              \varphi(x)=x+\lambda x^2,\\\\
           a_{1}^{(g_{\phi})}=-\frac{1}{2}d(d+1)\lambda, \\\\
             a_1^{(g_F)}=-\frac{1}{2}n(n+1)\lambda,
                     \end{array}
  \right.
\end{equation}
where $n=d+d_0$.

For the case of $d=1$, substituting \eqref{ee2.18} into \eqref{e5.55} and \eqref{e5.57}, we have
\begin{equation}\label{ee2.23}
  |R_{g_F}|^2-4|\mathrm{Ric}_{g_F}|^2=-2n(n+1)(2n+1)A^2+\frac{ |R_{g_{\phi}}|^2-4|\mathrm{Ric}_{g_{\phi}}|^2+12(d_0\lambda-nA)^2}{(1+\lambda x)^2},
\end{equation}
thus
\begin{equation}\label{ee2.24}
  \left\{\begin{array}{l}
           |R_{g_{\phi}}|^2-4|\mathrm{Ric}_{g_{\phi}}|^2=-12(d_0\lambda-nA)^2, \\\\
            |R_{g_F}|^2-4|\mathrm{Ric}_{g_F}|^2=-2n(n+1)(2n+1)A^2.
         \end{array}
  \right.
\end{equation}
So
\begin{equation}\label{ee2.25}
  \left\{\begin{array}{l}
           a_{2}^{(g_{\phi})}=0, \\\\
            a_2^{(g_F)}=\frac{1}{24}(n-1)n(n+1)(3n+2)A^2.
         \end{array}
  \right.
\end{equation}

For the case of $d>1$, by \eqref{ee2.19},  \eqref{e5.55} and \eqref{e5.57}, we get
\begin{equation}\label{ee2.20}
  |R_{g_F}|^2-4|\mathrm{Ric}_{g_F}|^2=-2n(n+1)(2n+1)\lambda^2+\frac{ |R_{g_{\phi}}|^2-4|\mathrm{Ric}_{g_{\phi}}|^2+2d(d+1)(2d+1)\lambda^2}{(1+\lambda x)^2},
\end{equation}
then
\begin{equation}\label{ee2.21}
  \left\{\begin{array}{l}
           |R_{g_{\phi}}|^2-4|\mathrm{Ric}_{g_{\phi}}|^2=-2d(d+1)(2d+1)\lambda^2, \\\\
            |R_{g_F}|^2-4|\mathrm{Ric}_{g_F}|^2=-2n(n+1)(2n+1)\lambda^2.
         \end{array}
  \right.
\end{equation}
Therefore
\begin{equation}\label{ee2.22}
  \left\{\begin{array}{l}
           a_{2}^{(g_{\phi})}=\frac{1}{24}(d-1)d(d+1)(3d+2)\lambda^2, \\\\
            a_2^{(g_F)}=\frac{1}{24}(n-1)n(n+1)(3n+2)\lambda^2.
         \end{array}
  \right.
\end{equation}

Using (3.9) of \cite{Feng2017}, it follows that
\begin{equation}\label{ee2.26}
  F(t)=\left\{\begin{array}{l}
                -\frac{1}{A}\log|1+ce^{t}|,\;A\neq 0,\\\\
                ce^{t},\;A=0, c>0;\\
              \end{array}
  \right.
\end{equation}

For the case of $$ M=\left\{(z,w)\in \Omega\times\mathbb{C}^{d_0}: \|w\|^2<e^{-\lambda\phi(z)}\right\}.$$
 The completeness of the fibre metric of $g_F$ requires
 \begin{equation*}
  \int_{-2\log 2}^{0}\sqrt{F''(t)}dt=+\infty.
\end{equation*}
Using \eqref{ee2.26}, we have
\begin{equation*}
  F(t)= -\frac{1}{A}\log\left(1-e^{t}\right).
\end{equation*}
From $g_F$ is a K\"{a}hler metric on $M$, it follows that $1+\lambda F'(t)>0$ and $F''(t)>0$ for $-\infty<t<0$. So by $\lambda\neq 0$, we get $A>0$ and $ \lambda>0$.

 For the case of $ M= \Omega\times\mathbb{C}^{d_0}$. If $A=0$, then $d=1$ and $F(t)=ce^t$, it is easy to see that the fibre metric of $g_F$ is complete at this time.
If $A\neq 0$, since $1+c e^{t}\neq 0$ for $-\infty\leq t<+\infty$, it follows that $c>0$. By \eqref{ee2.26}, we have
\begin{equation*}
  F(t)=- \frac{1}{A}\log\left(1+c e^{t}\right),c>0.
\end{equation*}
As $\lambda\neq 0$, $1+\lambda F'(t)>0$ and $F''(t)>0$ for $-\infty<t<+\infty$, it follows that $A<0$ and $ A\leq \lambda$.

\end{proof}

The following we give sufficient and  necessary conditions for some trivial Hermitian holomorphic vector bundles and their ball subbundles to be regular quantizations.

\begin{Theorem} \label{Th:1.2}{
Let $g_{\phi}$ be a K\"{a}hler  metric  on a domain $\Omega\subset \mathbb{C}^{d}$ associated with the K\"{a}hler form $\omega_{\phi}=\frac{\sqrt{-1}}{2\pi}\partial\overline{\partial}\phi$, where $\phi$ is a globally defined real analytic K\"{a}hler potential on $\Omega$. For $0\neq \lambda\in\mathbb{R}$, define
$$ M:=\left\{(z,w)\in \Omega\times\mathbb{C}^{d_0}: \|w\|^2<e^{-\lambda\phi(z)}\right\}\;\;\text{or}\;\;\Omega\times\mathbb{C}^{d_0}.$$
Set $g_F$ is a K\"{a}hler  metric  on the domain $M$ associated with the K\"{a}hler form $\omega_F=\frac{\sqrt{-1}}{2\pi}\partial\overline{\partial}\Phi_F$, here
$$\Phi_F(z,w):=\phi(z)+F(e^{\lambda\phi(z)}\|w\|^2),\;F(0)=0.$$

 Suppose that there exists a set $E\subset [0,+\infty)$ such that
 $$\mathcal{H}^2_{\alpha}(\Omega):=\left\{ f\in \mbox{\rm Hol}\big({\Omega}\big):  \int_{{\Omega}}|f|^2
 e^{-\alpha \phi} \frac{\omega^{d}_{\phi}}{d!}<+\infty\right\}\neq \{0\}$$
  for $\alpha\in E$, where $\mbox{\rm Hol}(\Omega)$ denotes the space of holomorphic functions on $\Omega$, and $E$ satisfies
\begin{equation*}
 E= E+\lambda\mathbb{N}:=\{\alpha+\lambda k:\alpha\in E,k\in \mathbb{N}\}
\end{equation*}
for $\lambda>0$.

Let $K_{\alpha}(z,\bar z)$ be the reproducing kernels of $\mathcal{H}^2_{\alpha}(\Omega)$, the Bergman functions $\epsilon_{\alpha g_{\phi}}$ for $(\Omega,\alpha g_{\phi})$ defined by $\epsilon_{\alpha g_{\phi}}=e^{-\alpha\phi(z)}K_{\alpha}(z,\bar z)$. In the same way, define the Bergman functions
$\epsilon_{\alpha g_F}$ for $(M,\alpha g_{F})$.

$(\mathrm{I})$ Given a number $\alpha_0\in E$, assume that the fibre metric of $g_F$ is complete, then Bergman functions $\epsilon_{\alpha g_F}$ are constants  on the domain
$$ M=\left\{(z,w)\in \Omega\times\mathbb{C}^{d_0}: \|w\|^2<e^{-\lambda\phi(z)}\right\}$$
 for all $\alpha\in E\cap [\alpha_0,+\infty)$ if and only if

$(\mathrm{i})$ For the case of $d=1$,
\begin{equation}\label{e2.1}
  \left\{\begin{array}{l}
    F(\rho)   =-\frac{1}{A}\log(1-\rho),\;A>0, \\\\
     \epsilon_{\alpha g_{\phi}}   =\alpha+d_0\lambda-nA, \lambda>0, \\\\
       \epsilon_{\alpha g_F} = \prod_{j=1}^{n}(\alpha-jA), \\\\
     \alpha\in (nA,+\infty)\cap [\alpha_0,+\infty)\cap E,
   \end{array}\right.
\end{equation}
where $n=d+d_0$.

$(\mathrm{ii})$ For the case of $d>1$,
\begin{equation}\label{e2.2}
  \left\{\begin{array}{l}
    F(\rho)   =-\frac{1}{\lambda}\log(1-\rho),\lambda>0, \\\\
    \epsilon_{\alpha g_{\phi}} = \prod_{j=1}^{d}(\alpha-j\lambda ), \\\\
       \epsilon_{\alpha g_F} = \prod_{j=1}^{n}(\alpha-j\lambda ), \\\\
    \alpha\in (nA,+\infty)\cap [\alpha_0,+\infty)\cap E.
   \end{array}\right.
\end{equation}

$(\mathrm{II})$ Let $\lambda=1$ and $\alpha_0\in E$, suppose that the fibre metric of $g_F$ is complete, then Bergman functions $\epsilon_{\alpha g_F}$ are constants  on the domain $M=\Omega\times\mathbb{C}^{d_0}$
for all $\alpha\in E\cap [\alpha_0,+\infty)$ if and only if
\begin{equation}\label{e2.3}
  \left\{\begin{array}{l}
            d=1,\\\\
            F(\rho)=c\rho,c>0, \\\\
    \epsilon_{\alpha g_{\phi}} =\alpha+d_0, \\\\
       \epsilon_{\alpha g_F} = \alpha^n, \\\\
     \alpha\in   (0,+\infty)\cap [\alpha_0,+\infty)\cap E.
   \end{array}\right.
\end{equation}

$(\mathrm{III})$ For $d>1$, $\lambda=-1$ and $E=\mathbb{N}$, let $\alpha_0\in \mathbb{N}$, then Bergman functions $\epsilon_{\alpha g_F}$ are constants on the domain $M=\Omega\times\mathbb{C}^{d_0}$
for all $\alpha\in \mathbb{N}\cap [\alpha_0,+\infty)$ if and only if
\begin{equation}\label{e2.4}
  \left\{\begin{array}{l}
    F(\rho)   =\log(1+c\rho), c>0,\\\\
    \epsilon_{\alpha g_{\phi}} = \prod_{j=1}^{d}(\alpha+j), \\\\
       \epsilon_{\alpha g_F} = \prod_{j=1}^{n}(\alpha+j), \\\\
     \alpha\in \mathbb{N}\cap [\alpha_0,+\infty).
   \end{array}\right.
\end{equation}

 }\end{Theorem}

\begin{proof}[Proof]
The following we only prove necessary conditions,  sufficiency conditions are obvious.

As $\epsilon_{\alpha g_F}$ are constants on $M$ for all $\alpha\in E\cap[\alpha_0,+\infty)$,  thus
\begin{equation*}
  \epsilon_{\alpha g_F}\sim \alpha^n+ a_1^{(g_F)}\alpha^{n-1}+ a_2^{(g_F)}\alpha^{n-2}+\cdots,\;\alpha\rightarrow+\infty,
\end{equation*}
where $ a_1^{(g_F)}$ and $ a_2^{(g_F)}$ are constants, and
\begin{equation*}
   a_1^{(g_F)}=\frac{1}{2}k_{g_F},\; a_2^{(g_F)}=\frac{1}{3}\triangle_{g_F} k_{g_F}+\frac{1}{24}|R_{g_F}|^2-\frac{1}{6}|\mathrm{Ric}_{g_F}|^2+\frac{1}{8}k_{g_F}^2.
\end{equation*}

By Theorem \ref{apth:3.3}, we have
\begin{equation*}
  F(\rho)=\left\{\begin{array}{c}
                -\frac{1}{A}\log\left(1-\rho\right),d=1, \\\\
                -\frac{1}{\lambda}\log\left(1-\rho\right),d>1
              \end{array}
  \right.
\end{equation*}
for $ M=\left\{(z,w)\in \Omega\times\mathbb{C}^{d_0}: \|w\|^2<e^{-\lambda\phi(z)}\right\}$, and
\begin{equation*}
  F(\rho)=\left\{\begin{array}{l}
                c\rho,\;d=1,c>0, \\\\
               \log\left(1+c\rho\right),\;d>1,c>0
              \end{array}
  \right.
\end{equation*}
for $ M= \Omega\times\mathbb{C}^{d_0}$.

 Put
$$\mathbf{H}_{\alpha}^2(M):=\left\{ f\in \mbox{\rm Hol}\big( M\big): \frac{1}{\pi^n}\int_{ M}|f(Z)|^2
 e^{-\alpha \Phi_{F}(Z)} \det(\frac{\partial^2\Phi_{F}}{\partial Z^t\partial \overline{Z}})(Z)dm(Z)<+\infty\right\},$$
where $n=d+d_0$, $Z=(z,w)$, and $dm(Z)$ denotes the Lebesgue measure.

Let $f\in \mathbf{H}_{\alpha}^2(M)$, then
\begin{equation}\label{e6.4}
 f(z,w)=\sum_{\mathbf{m}}f_{\mathbf{m}}(z)w^{\mathbf{m}},
\end{equation}
where $\mathbf{m}=(m_1,m_2,\ldots,m_{d_0})$, $m_i\in \mathbb{N}$, $1\leq i\leq d_0$, $w^{\mathbf{m}}=\prod_{i=1}^{d_0}w_i^{m_i}$.

From Lemma 2.1 of \cite{Feng}, we get
\begin{equation*}
  e^{-\alpha \Phi_{F}(Z)} \det(\frac{\partial^2\Phi_{F}}{\partial Z^t\partial \overline{Z}})(Z)
  =e^{-(\alpha-\lambda d_0) \phi(z)}\det(\frac{\partial^2\phi}{\partial z^t\partial \bar z})(z) H(\alpha,\rho),
\end{equation*}
here
\begin{equation*}
   H(\alpha,\rho)=e^{-\alpha F(\rho)}(F'(\rho))^{d_0-1}(F'(\rho)+\rho F''(\rho))(1+\lambda\rho F'(\rho))^{d}.
\end{equation*}

Using the measures $ e^{-\alpha \Phi_{F}(Z)} \det(\frac{\partial^2\Phi_{F}}{\partial Z^t\partial \overline{Z}})(Z)dm(Z)$ are invariant under transformations
 $$(z,w_1,w_2,\ldots,w_{d_0})\mapsto \left(z,e^{\sqrt{-1}\theta_{1}}w_1,e^{\sqrt{-1}\theta_{2}}w_2,\ldots,e^{\sqrt{-1}\theta_{d_0}}w_{d_0}\right),\;\forall\theta_i\in\mathbb{R},1\leq i\leq d_0,$$
we have
\begin{eqnarray}
\nonumber   & & \frac{1}{\pi^n}\int_{ M}|f(Z)|^2 e^{-\alpha \Phi_{F}(Z)} \det(\frac{\partial^2\Phi_{F}}{\partial Z^t\partial \overline{Z}})(Z)dm(Z) \\
\nonumber   &=& \frac{1}{\pi^n}\sum_{\mathbf{m}}\int_{ M}|f_{\mathbf{m}}(z)|^2w^{\mathbf{m}}\overline{w}^{\mathbf{m}} e^{-\alpha \Phi_{F}(Z)} \det(\frac{\partial^2\Phi_{F}}{\partial Z^t\partial \overline{Z}})(Z)dm(Z) \\
\label{e6.6.1}    &=& \sum_{\mathbf{m}}\frac{1}{\pi^{d}}\int_{\Omega}|f_{\mathbf{m}}(z)|^2e^{-(\alpha+\lambda|\mathbf{m}|)\phi(z)}\times \det(\frac{\partial^2\phi}{\partial z^t\partial \bar z})(z)dm(z)\times I_{\mathbf{m}},
\end{eqnarray}
here $|\mathbf{m}|=\sum_{i=1}^{d_0}m_i$ and
\begin{equation}\label{e6.7.1}
 I_{\mathbf{m}}=\frac{1}{\pi^{d_0}} \int_{\mathbb{B}^{d_0}}w^{\mathbf{m}}\overline{w}^{\mathbf{m}}H(\alpha,\|{w}\|^{2}) dm(w)\;\text{or}\;\;\frac{1}{\pi^{d_0}} \int_{\mathbb{C}^{d_0}}w^{\mathbf{m}}\overline{w}^{\mathbf{m}}H(\alpha,\|{w}\|^{2}) dm(w).
\end{equation}

Let $d\sigma$ be the Euclidean invariant measure on the  sphere $\mathbb{S}^{2d_0-1}$, using integral formulas
\begin{equation}\label{e6.8.1}
 \int_{\mathbb{S}^{2d_0-1}}w^{\mathbf{m}}\overline{w}^{\mathbf{m}}d\sigma(w)=\frac{2\pi^{d_0}\prod_{i=1}^{d_0}\Gamma(1+m_i)}{\Gamma(|\mathbf{m}|+d_0)},
\end{equation}
we obtain
\begin{equation}\label{e6.9.1}
  I_{\mathbf{m}}=\frac{\prod_{i=1}^{d_0}\Gamma(1+m_i)}{\Gamma(|\mathbf{m}|+1)}\psi(\alpha,|\mathbf{m}|),
\end{equation}
where
\begin{equation}\label{e6.10.1}
  \psi(\alpha,k)=\frac{\Gamma(k+1)}{\Gamma(k+d_0)} \int_0^1u^{k+d_0-1}H(\alpha,u)du\;\;\text{or}\;\;\frac{\Gamma(k+1)}{\Gamma(k+d_0)} \int_0^{+\infty}u^{k+d_0-1}H(\alpha,u)du.
\end{equation}

According to \eqref{e6.6.1} and \eqref{e6.9.1}, we get the orthogonal direct sum decompositions
\begin{equation*}
 \mathbf{H}_{\alpha}^2(M)=\bigoplus_{I_{\mathbf{m}}<+\infty,\atop \alpha+\lambda|\mathbf{m}|\in E} \mathcal{H}^2_{\alpha+\lambda|\mathbf{m}|}(\Omega)\otimes \{cw^{\mathbf{m}}:c\in \mathbb{C}\}
\end{equation*}
for $\alpha\in E$. Then the reproducing kernels of  $\mathbf{H}_{\alpha}^2(M)$ can be expressed as
\begin{eqnarray}
\nonumber    \mathbf{K}_{\alpha}(Z,\overline{Z})    &= & \sum_{I_{\mathbf{m}}<+\infty,\atop \alpha+\lambda|\mathbf{m}|\in E} K_{\alpha+\lambda|\mathbf{m}|}(z,\bar{z})\frac{1}{\psi(\alpha,|\mathbf{m}|)}\frac{\Gamma(|\mathbf{m}|+1)}{\prod_{i=1}^{d_0}\Gamma(1+m_i)}w^{\mathbf{m}}\overline{w}^{\mathbf{m}} \\
 \label{e6.11.1}  &=&\sum_{k\in \mathbb{N},\atop\alpha+\lambda k\in E}K_{\alpha+\lambda k}(z,\bar{z},)\frac{1}{\psi(\alpha,k)}\|w\|^{2k},
\end{eqnarray}
where $K_{\alpha}(z,\bar{z})$ are the reproducing kernels of $\mathcal{H}^2_{\alpha}(\Omega)$.

Applying $K_{\alpha}(z,\bar{z})=e^{\alpha\phi(z)} \epsilon_{\alpha g_{\phi}}(z) $,
$\mathbf{K}_{\alpha}(z,w,\overline{z},\overline{w})=e^{\alpha\Phi_F(z,w)}\epsilon_{\alpha g_F}(z,w)$ and
\eqref{e6.11.1}, we have
\begin{equation}\label{e6.12.1}
\epsilon_{\alpha g_F}(z,w)=e^{-\alpha F(\rho)} \sum_{k\in \mathbb{N},\atop\alpha+\lambda k\in E}\frac{\epsilon_{(\alpha+\lambda k)g_{\phi}}(z)}{\psi(\alpha,k)}\rho^{k},
\end{equation}
where
\begin{equation*}
 \rho=e^{\lambda\phi(z)}\|w\|^2.
\end{equation*}

Talking $w=0$ in \eqref{e6.12.1}, using $F(0)=0$, we get
\begin{equation}\label{e6.13.1}
 \epsilon_{\alpha g_F}(z,0)=\frac{\epsilon_{\alpha g_{\phi}}(z)}{\psi(\alpha,0)},
\end{equation}
by $\epsilon_{\alpha g_F}(z,w)$ are constants on $M$, it follows that $\epsilon_{\alpha g_{\phi}}(z)$ are constants on $\Omega$.

Substituting \eqref{e6.13.1} into \eqref{e6.12.1}, we have
\begin{equation}\label{e6.14.1}
e^{\alpha F(\rho)} =\sum_{k\in \mathbb{N},\atop\alpha+\lambda k\in E}\frac{\epsilon_{(\alpha+\lambda k)g_{\phi}}}{\epsilon_{\alpha g_{\phi}}}\frac{\psi(\alpha,0)}{\psi(\alpha,k)}\rho^k,\;\alpha\in[\alpha_0,+\infty)\cap E.
\end{equation}

 $(\mathrm{I})$ For
 $$ M=\left\{(z,w)\in \Omega\times\mathbb{C}^{d_0}: \|w\|^2<e^{-\lambda\phi(z)}\right\},$$
using
\begin{equation*}
 F(\rho)=\left\{\begin{array}{l}
               -\frac{1}{A}\log\left(1-\rho\right),\;d=1, \\\\
               -\frac{1}{\lambda}\log\left(1-\rho\right),\;d>1,
             \end{array}
 \right.
\end{equation*}
then \eqref{e6.14.1} becomes
\begin{equation}\label{e6.14}
\sum_{k=0}^{+\infty}\frac{\epsilon_{(\alpha+\lambda k)g_{\phi}}}{\epsilon_{\alpha g_{\phi}}}\frac{\psi(\alpha,0)}{\psi(\alpha,k)}\rho^k=\left\{\begin{array}{l}
               \left(1-\rho\right)^{-\frac{\alpha}{A}},\;d=1, \\\\
               \left(1-\rho\right)^{-\frac{\alpha}{\lambda}},\;d>1,
             \end{array}
 \right.
\end{equation}
where
\begin{equation}\label{e6.10}
  \psi(\alpha,k)=\left\{\begin{array}{l}
                          \frac{\Gamma(k+1)\Gamma(\frac{\alpha}{A}-n)\times(\alpha+\lambda k+d_0\lambda-nA)}{A^{n}\Gamma(\frac{\alpha}{A}+k)},
  \;\alpha\in (nA,+\infty)\cap E,d=1, \\\\
                          \frac{\Gamma(k+1)\Gamma(\frac{\alpha}{\lambda}-n)}{\lambda^{d_0}\Gamma(\frac{\alpha}{\lambda}+k-d)},\;\alpha\in (n\lambda,+\infty)\cap E,d>1.
                        \end{array}
  \right.
\end{equation}

Using \eqref{e6.14} and
\begin{equation*}
  (1-u)^{-\alpha}=\sum_{k=0}^{+\infty}\frac{\Gamma(\alpha+k)}{\Gamma(\alpha)\Gamma(k+1)}u^k,\;|u|<1,
\end{equation*}
 we get
\begin{equation*}
 \frac{\epsilon_{(\alpha+\lambda k)g_{\phi}}}{\epsilon_{\alpha g_{\phi}}}\frac{\psi(\alpha,0)}{\psi(\alpha,k)}=\left\{\begin{array}{c}
                                                                                                                  \frac{\Gamma(\frac{\alpha}{A}+k)}{\Gamma(\frac{\alpha}{A})\Gamma(k+1)},d=1,  \\\\
                                                                                                                    \frac{\Gamma(\frac{\alpha}{\lambda}+k)}{\Gamma(\frac{\alpha}{\lambda})\Gamma(k+1)},d>1.
                                                                                                                 \end{array}
 \right.
\end{equation*}
 Which combines with \eqref{e6.10}, we obtain
\begin{equation}\label{e6.16}
  \frac{\epsilon_{(\alpha+\lambda k)g_{\phi}}}{\epsilon_{\alpha g_{\phi}}}
 =\left\{\begin{array}{l}
           \frac{\alpha+\lambda k+d_0\lambda-nA}{\alpha+d_0\lambda-nA},d=1, \\\\
           \frac{\Gamma(\frac{\alpha}{\lambda}+k)}{\Gamma(\frac{\alpha}{\lambda}+k-d)}\frac{\Gamma(\frac{\alpha}{\lambda}-d)}{\Gamma(\frac{\alpha}{\lambda})},d>1.
         \end{array}
 \right.
\end{equation}

Combining \eqref{e6.16} and
$$\lim_{k\rightarrow+\infty} \frac{\epsilon_{(\alpha+\lambda k)g_{\phi}}}{(\lambda k)^{d}\epsilon_{\alpha g_{\phi}}}
= \frac{1}{\epsilon_{\alpha g_{\phi}}},$$
it follows that
\begin{equation}\label{e6.17}
\epsilon_{\alpha g_{\phi}}=\left\{\begin{array}{l}
                                     \alpha+d_0\lambda-nA,d=1,  \\\\
                                     \frac{\lambda^d\Gamma(\frac{\alpha}{\lambda})}{\Gamma(\frac{\alpha}{\lambda}-d)},d>1.
                                    \end{array}
\right.
\end{equation}

Using \eqref{e6.12.1}, \eqref{e6.10} and \eqref{e6.17}, we have
\begin{equation*}
  \epsilon_{\alpha g_F}=\left\{\begin{array}{l}
                                  \frac{ A^n\Gamma(\frac{\alpha}{A})}{\Gamma(\frac{\alpha}{A}-n)},\;d=1,  \\\\
                                   \frac{\lambda^n\Gamma(\frac{\alpha}{\lambda})}{\Gamma(\frac{\alpha}{\lambda}-n)},\;d>1
                                 \end{array}
  \right.
\end{equation*}
for $\alpha\in (nA,+\infty)\cap[\alpha_0,+\infty)\cap E$.

$(\mathrm{II})$ For $M= \Omega\times\mathbb{C}^{d_0}$, $d=1$, $c>0$ and $\lambda=1$, substituting
\begin{equation*}
 F(\rho)=c\rho
\end{equation*}
 into  \eqref{e6.14.1}, we obtain
\begin{equation}\label{e6.14.2}
\sum_{k=0}^{+\infty}\frac{\epsilon_{(\alpha+\lambda k)g_{\phi}}}{\epsilon_{\alpha g_{\phi}}}\frac{\psi(\alpha,0)}{\psi(\alpha,k)}\rho^k
=e^{c\alpha \rho},
\end{equation}
where
\begin{equation}\label{e6.10.1}
  \psi(\alpha,k)=\frac{\Gamma(k+1)\times(\alpha+\lambda k+\lambda d_0)}{c^k\alpha^{k+d_0+1}},
  \;\alpha\in  E\cap (0,+\infty).
\end{equation}

From \eqref{e6.14.2} and \eqref{e6.10.1}, we infer that
\begin{equation}\label{e6.16.1}
  \frac{\epsilon_{(\alpha+\lambda k)g_{\phi}}}{\epsilon_{\alpha g_{\phi}}}
 = \frac{\alpha+\lambda k+\lambda d_0}{\alpha+\lambda d_0},\alpha\in  E\cap (0,+\infty)\cap[\alpha_0,+\infty),
\end{equation}

By
\begin{equation*}
  \epsilon_{\alpha g_{\phi}}\sim \alpha,\alpha\rightarrow+\infty,
\end{equation*}
we get
\begin{equation*}
 \epsilon_{\alpha g_{\phi}}=\alpha+\lambda d_0=\alpha+d_0,
\end{equation*}
which combines with \eqref{e6.12.1} and \eqref{e6.10.1}, it follows that
\begin{equation*}
   \epsilon_{\alpha g_F}=\alpha^{1+d_0},\alpha\in  E\cap (0,+\infty)\cap[\alpha_0,+\infty).
\end{equation*}

$(\mathrm{III})$ For $M= \Omega\times\mathbb{C}^{d_0}$ and $\lambda=-1$, using
\begin{equation*}
 F(\rho)=\log(1+c\rho), \;c>0,
\end{equation*}
we infer from \eqref{e6.14.1} that
\begin{equation}\label{e6.14.3}
\sum_{k=0}^{\alpha}\frac{\epsilon_{(\alpha-k)g_{\phi}}}{\epsilon_{\alpha g_{\phi}}}\frac{\psi(\alpha,0)}{\psi(\alpha,k)}\rho^k
=(1+c\rho)^{\alpha},\;\alpha\in E\cap[\alpha_0,+\infty),
\end{equation}
where
\begin{equation}\label{e6.10.2}
  \psi(\alpha,k)=\frac{\Gamma(k+1)\Gamma(\alpha-k+d+1)}{c^k\Gamma(\alpha+n+1)}, \;\alpha\in  E, 0\leq k\leq \alpha.
\end{equation}

Using \eqref{e6.14.3} and \eqref{e6.10.2}, we have
\begin{equation*}
  \frac{\epsilon_{(\alpha-k)g_{\phi}}}{\epsilon_{\alpha g_{\phi}}}
 = \frac{\prod_{j=1}^d(\alpha-k+j)}{\prod_{j=1}^d(\alpha+j)},\;\alpha\in E\cap[\alpha_0,+\infty),0\leq k\leq \alpha.
\end{equation*}
Namely
\begin{equation}\label{e6.16.2}
  \frac{\epsilon_{kg_{\phi}}}{\epsilon_{\alpha g_{\phi}}}
 = \frac{\prod_{j=1}^d(k+j)}{\prod_{j=1}^d(\alpha+j)},\;\alpha\in E\cap[\alpha_0,+\infty),0\leq k\leq \alpha.
\end{equation}

According to \eqref{e6.16.2} and
$$\lim_{k\rightarrow+\infty} \frac{\epsilon_{\alpha g_{\phi}}}{\alpha^{d}}= 1,$$
we have
\begin{equation*}
 \epsilon_{k g_{\phi}}=\prod_{j=1}^d(k+j).
\end{equation*}
So from \eqref{e6.12.1} and \eqref{e6.10.2},   we obtain
\begin{equation*}
   \epsilon_{\alpha g_F}=\prod_{j=1}^n(\alpha+j),\alpha\in  E\cap[\alpha_0,+\infty).
\end{equation*}

\end{proof}

\setcounter{equation}{0}
\section{Proof of Theorem \ref{Th:1.3}}

In order to prove Theorem \ref{Th:1.3}, we need the following lemma.

\begin{Lemma} \label{Le:2.2}{
Let $(L, h)$ be a Hermitian holomorphic line bundle over a K\"{a}hler manifold $(M,\omega)$ of complex dimension $n$.  Let $f$ be a holomorphic section on $M- Y$ such that
\begin{equation*}
  \int_{M-Y}h(f(z),f(z))\frac{\omega^n}{n!}<+\infty,
\end{equation*}
where $Y$ is an analytic set of $M$. Then $f$ extends to a (unique) global holomorphic section, namely there exists $s\in H^0(M,L)$ such that $s(z)=f(z)$ for all $z\in M- Y$.

 }\end{Lemma}

 Lemma \ref{Le:2.2} can be proved by using Skoda's lemma (see e.g. Lemma 2.3.22 of \cite{MM07}) and its application (\cite{MM07}, (ii) in Lemma 6.2.1), for details refer to  Lemma 4.1 of \cite{Loi-Mossa-Zuddas}.

\begin{proof}[\textbf{Proof of Theorem} \ref{Th:1.3}]
The part $(\mathrm{II})$  can be proved in the same way as  the part $(\mathrm{I})$, so only the proof of the part $(\mathrm{I})$ is given.

  Suppose  $U$ is biholomorphically equivalent to a domain $\Omega\subset \mathbb{C}^d$, and
\begin{equation*}
  \theta_0:L_0\supset \pi_0^{-1}(U)\rightarrow U\times \mathbb{C}
\end{equation*}
is an analytic homeomorphism such that for every $z\in U$ the map
\begin{equation*}
 \theta_0: \pi_0^{-1}(z) \rightarrow  \{z\}\times \mathbb{C}\rightarrow \mathbb{C}
\end{equation*}
is a linear isomorphism.

Let $\tau(z):=\theta_0^{-1}(z,1)\;(z\in U)$ be a trivializing holomorphic section,  then
$$\pi_0^{-1}(U)=\{\mathbf{v}=u\tau(z)\in L_0:z\in U,u\in \mathbb{C}\},$$
and there exists a globally defined real function  $\phi$ on $U$ such that
\begin{equation*}
  h_0(\mathbf{v},\mathbf{v})=e^{-\phi(z)}|u|^2,\;\mathbf{v}=u\tau(z)\in \pi^{-1}(U).
\end{equation*}

For the vect bundle $L_0^{*\oplus r}$ over $M_0$, since
\begin{equation*}
  \theta_p:\left\{\begin{array}{rcl}
             L_0^{*\oplus r}\supset p^{-1}(U) & \longrightarrow & U\times\mathbb{C}^r, \\\\
             (z,w\tau^{*}(z)) & \longmapsto & (z,w)
           \end{array}\right.
\end{equation*}
is an analytic homeomorphism, it follows that
\begin{equation*}
  \theta_{p\circ\Pi}:\left\{\begin{array}{rclll}
             L\supset (p\circ\Pi)^{-1}(U) & \longrightarrow & p^{-1}(U)\times\mathbb{C}& \longrightarrow & U\times\mathbb{C}^{r}\times\mathbb{C}, \\\\
             (z,w\tau^{*}(z),v\tau(z)) & \longmapsto & (z,w\tau^{*}(z),v)& \longmapsto & (z,w,v)
           \end{array}\right.
\end{equation*}
 is also an analytic homeomorphism, where $\tau^{*}(z)$ are the duals of $\tau(z)$. So
\begin{equation*}
  h(\mathbf{x},\mathbf{x})=e^{\phi(z)}\|w\|^2,\;\mathbf{x}=(z,w\tau^{*}(z))\in p^{-1}(U),
\end{equation*}
\begin{equation*}
M\cap p^{-1}(U)= B(L_0^{*\oplus r})\cap p^{-1}(U)=\left\{(z,w\tau^{*}(z)):z\in U,w\in\mathbb{C}^r, e^{\phi(z)}\|w\|^2<1\right\},
\end{equation*}
\begin{equation*}
  h_F(\mathbf{y},\mathbf{y})=e^{-F(e^{\phi(z)}\|w\|^2)}e^{-\phi(z)}|v|^2,\;\mathbf{y}=(z,w\tau^{*}(z),v\tau(z))\in (p\circ\Pi)^{-1}(U)
\end{equation*}
and
\begin{equation*}
  \omega_F|_{M\cap p^{-1}(U)}=\frac{\sqrt{-1}}{2\pi}\partial\overline{\partial}\Phi_F,\;\Phi_F(z,w\tau^{*}(z))=\phi(z)+F(e^{\phi(z)}\|w\|^2).
\end{equation*}

For any $s\in H^{0}(M,L^m)$, then
\begin{equation*}
  s|_{M\cap p^{-1}(U)}(\mathbf{x})=f(z,w)\tau^m(z),\;\mathbf{x}=(z,w\tau^{*}(z))\in p^{-1}(U),
\end{equation*}
where $f$ is a holomorphic function on a domain
\begin{equation*}
  \widetilde{U}=\left\{(z,w)\in U\times\mathbb{C}^r:e^{\phi(z)}\|w\|^2<1\right\}.
\end{equation*}

Let
\begin{equation*}
  \mathcal{H}_m^2M)=\left\{s\in H^{0}(M,L^m):\int_{M}h_F^m(s,s)\frac{\omega_F^{d+r}}{(d+r)!}<+\infty\right\}
\end{equation*}
and
\begin{equation*}
\mathcal{H}_{m}^2(\widetilde{U})=\left\{ f\in \mbox{\rm Hol}\big({\widetilde{U}}\big): \int_{\widetilde{U}}|f(Z)|^2
 e^{-m \Phi_F(Z)}\frac{\varpi_F^{d+r}(Z)}{(d+r)!}<+\infty\right\},
\end{equation*}
where $ \varpi_F=\frac{\sqrt{-1}}{2\pi}\partial\overline{\partial}\widetilde{\Phi}_F$, $\widetilde{\Phi}_F(Z)=\phi(z)+F(e^{\phi(z)}\|w\|^2)$ for $Z=(z,w)\in \widetilde{U}$, and $\mbox{\rm Hol}\big({\widetilde{U}}\big)$ stands for the space of holomorphic functions on $\widetilde{U}$.

Since $U$ is a dense open contractible subset of $M_0$, it follows that the correspondence
$$\widetilde{U}\ni (z, w) \longmapsto (z, w\tau^{*}(z))\in M $$
 sets up a bijection between the Hartogs domain $\widetilde{U}$ and a dense open subset of $M$ . By $p^{-1}(M_0- U)$ is an analytic set and Lemma \ref{Le:2.2}, there exists an isometric isomorphism between $\mathcal{H}_m^2(M)$ and $ \mathcal{H}_{m}^2(\widetilde{U})$, that is
\begin{equation*}
  s\in \mathcal{H}_m^2(M)\longmapsto f\in \mathcal{H}_{m}^2(\widetilde{U})
\end{equation*}
if $s|_{M\cap p^{-1}(U)}(\mathbf{x})=f(z,w)\tau^m(z)$, for $\mathbf{x}=(z,w\tau^{*}(z))\in p^{-1}(U)$. Thus
\begin{equation}\label{e2.4}
  \epsilon_{m g_F}(z,w\tau^{*}(z))=\epsilon_{m\widetilde{g}_F}(z,w),\;(z,w)\in\widetilde{U},
\end{equation}
here $\widetilde{g}_F$ is a K\"{a}hler  metric  on $\widetilde{U}$ associated to the K\"{a}hler form $\varpi_F=\frac{\sqrt{-1}}{2\pi}\partial\overline{\partial}\widetilde{\Phi}_F$. Similarly, we have
\begin{equation}\label{e2.5}
   \epsilon_{m g_0}(z)=\epsilon_{mg_{\phi}}(z),\;z\in U,
\end{equation}
where $g_{\phi}$ is a K\"{a}hler  metric  on $U$ associated to the K\"{a}hler form $\omega_{\phi}=\frac{\sqrt{-1}}{2\pi}\partial\overline{\partial}\phi$.

As $\epsilon_{m g_0}$ and $\epsilon_{m g_F}$ are continuous functions, by $U$ and $M\cap p^{-1}(U)$ are  dense open subsets of $M_0$ and  $M$, respectively, it follows that both $\epsilon_{mg_{\phi}}$ and $\epsilon_{m\widetilde{g}_F}$ are constants if and only if  both $\epsilon_{m g_0}$ and $\epsilon_{m g_F}$ are constants. So using Theorem \ref{Th:1.2}, \eqref{e2.4} and \eqref{e2.5},  we obtain \eqref{e1.12} and \eqref{e1.13}.

\end{proof}

\setcounter{equation}{0}
\section{Proof of Theorem \ref{Th:1.4} }

In this section, using Theorem \ref{apth:3.3}, Theorem \ref{Th:1.2}, the Hirzebruch-Riemann-Roch formula and Kobayashi-Ochiai's characterization of the projective spaces, we give a proof of  Theorem \ref{Th:1.4}.

\begin{proof}[\textbf{Proof of Theorem} \ref{Th:1.4}]

Let $n=d+r$. Denote $g_0$  as a K\"{a}hler  metric  associated with the K\"{a}hler form $\omega_0$.

Because $(M,\omega)$ admits the regular quantization, from \cite{Loi1}, we obtain
\begin{equation*}
  \epsilon_{mg}\sim\sum_{j=0}^{+\infty} a_j^{(g)}m^{n-j},m\rightarrow +\infty,
\end{equation*}
and all $a_j^{(g)}$ are constants on $M$. Thus all $a_j^{(g)}$ are constants on $E$, using Theorem \ref{apth:3.3}, it follows that $a_j^{(g_0)}(1\leq j\leq 2)$ are constants on $M_0$,
\begin{equation*}
  F(\rho)=\left\{\begin{array}{l}
                -\frac{1}{A}\log(1+c\rho), \lambda\geq A, d=1, \\\\
                -\frac{1}{\lambda}\log(1+c\rho),d>1
              \end{array}
  \right.
\end{equation*}
and
\begin{equation*}
 a_1^{(g_0)}=\left\{\begin{array}{l}
                d_0\lambda-n A, d=1, \\\\

                -\frac{1}{2}d(d+1)\lambda, d>1,
              \end{array}
  \right.
\end{equation*}
where $c>0$ and $\lambda=-1$. According to  Lemma 4.1 of \cite{Fu-Yau-Zhou} and the explicit expression of function $F$, $\omega$ can be extended across $M-E$.

For $d=1$, the Ricci form $\mathrm{Ric_{g_0}}>0$ on the basis of $a_1^{(g_0)}$ is a constant on $M_0$ and $$a_1^{(g_0)}=d_0\lambda-n A=-d_0-n A\geq -d_0+n=d=1,$$
 it follows from \cite{Kobayashi-1961} that $M_0$ is simply connected. Using the uniformization theorem, we infer that $M_0$ is biholomorphic to a complex projective space $\mathbb{CP}^1$.

 Due to
\begin{equation*}
  \mathrm{Ric_{g_0}}=2a_1^{(g_0)}c_1(L_0,h_0),
\end{equation*}
it follows that
\begin{equation*}
  2\leq 2a_1^{(g_0)}=\frac{\int_{M_0}\mathrm{Ric_{g_0}}}{\int_{M_0}c_1(L_0,h_0)}=\frac{\int_{M_0}c_1(M_0)}{\int_{M_0}c_1(L_0)}=\frac{2}{k},k\geq 1,k\in\mathbb{N},
\end{equation*}
where $c_1(L_0)$ and $c_1(M_0)$ denote Chern classes of $L_0$ and $M_0$, respectively. So $a_1^{(g_0)}=1$ and $A=-1$.

For $d>1$, it follows from Theorem \ref{Th:1.2} that
\begin{equation*}
  \sum_{j=0}^{+\infty}\left.a_j^{(g_0)}\right|_{\Omega}\;m^{d-j}=\prod_{j=1}^d(m+j),m\in\mathbb{N}.
\end{equation*}
Thanks to $\Omega$ is a dense subset of $M_0$, then
\begin{equation*}
  \sum_{j=0}^{+\infty}a_j^{(g_0)}m^{d-j}=\prod_{j=1}^d(m+j),m\in\mathbb{N}.
\end{equation*}

Now recall the Hirzebruch-Riemann-Roch formula. Let $V$ be a holomorphic vector bundle on a compact complex manifold $X$.  Let $H^j(X,V)$ denote the j-th cohomology of  $X$ with coefficients in the sheaf $\mathcal{O}(V)$ of germs of holomorphic sections of $V$. Let
$$\chi(X,V):=\sum_{j=0}^{+\infty}(-1)^j\dim H^j(X,V)$$
be the Euler-Poincar\'{e} characteristic of a holomorphic vector bundle $V$ over a compact complex manifold $X$. The Hirzebruch-Riemann-Roch formula is given by
\begin{equation*}
  \chi(X,V)=\int_X \textrm{Td}(X) \textrm{ch}(V),
\end{equation*}
where  the total Todd class $\textrm{td}(X)$ is defined by
\begin{equation*}
  \textrm{Td}(X):=\prod_j\frac{\xi_j}{1-e^{-\xi_j}},\;\text{if}\;\sum_jc_j(X)x^j=\prod_j(1+\xi_jx),
\end{equation*}
and the total Chern character $\textrm{ch}(E)$ is given by
\begin{equation*}
  \textrm{ch}(V)=\sum_je^{\zeta_j}, \;\text{if}\;\sum_jc_j(V)x^j=\prod_j(1+\zeta_jx).
\end{equation*}
The above $c_j(V)$ are Chern classes of $V$, and  $c_j(X)$ are Chern classes of the holomorphic tangent bundle of $X$.

For $X=M_0$ and $V=L_0^m$, from the Hirzebruch-Riemann-Roch formula, it follows that $\chi(M_0,L_0^m)$ is a polynomial in $m$. Since $L_0$ is a positive holomorphic line bundle
 over the Fano manifold $M_0$, it follows that
\begin{equation*}
  c_1(K_{M_0}\otimes L_0^{-m})=c_1(K_{M_0})+c_1(L_0^{-m})=-c_1(M_0)-mc_1(L_0)<0
\end{equation*}
for $m\geq 0$, where $K_{M_0}$ is the canonical line bundle of $M_0$. By the Kodaira vanishing theorem (e.g., page 68 of  \cite{Kobayashi-1987}), we get
\begin{equation*}
\dim H^j(M_0,L_0^m)=\dim H^{d-j}(M_0,K_{M_0}\otimes L_0^{-m})=0
\end{equation*}
for $1\leq j\leq d$ and $m\geq 0$.  So $\dim H^0(M_0,L_0^m)$ is a polynomial of degree $d$  in $m$ for $m\geq 0$.

From
\begin{equation*}
  \epsilon_{mg_0}\sim\sum_{j=0}^{+\infty} a_j^{(g_0)}m^{d-j}=\prod_{j=1}^d(m+j),\;\text{as}\;m\rightarrow +\infty,
\end{equation*}
we have
\begin{equation*}
 \dim H^0(M_0,L_0^m)=\int_{M_0} \epsilon_{mg_0}\frac{\omega_0^d}{d!}\sim \sum_{j=0}^{\infty} a_j^{(g_0)}m^{d-j}\int_{M_0} \frac{\omega_0^d}{d!}\;\text{as}\;m\rightarrow +\infty.
\end{equation*}
Hence
\begin{equation*}
 \dim H^0(M_0,L_0^m)=\prod_{j=1}^d(m+j)\int_{M_0}\frac{\omega_0^d}{d!}
\end{equation*}
for $m\geq 0$.

By $\dim H^0(M_0,L_0^m)=1$ for $m=0$, it follows that
\begin{equation*}
 \int_{M_0}\omega_0^d=\int_{M_0}(c_1(L_0))^d=1\;\text{and}\;\dim H^0(M_0,L_0^m)=\frac{1}{d!}\prod_{j=1}^d(m+j).
\end{equation*}
According to Theorem 1.1 of \cite{Kobayashi-Ochiai}, $M_0$ is biholomorphic to the complex projective space $\mathbb{CP}^d$.

Using Corollary 4.2 in \cite{Arezzo-Loi},  there is a positive integer $j$ and an automorphism $\Upsilon\in \textrm{Aut}(M_0)=\textrm{PGL}(d+1,\mathbb{C})$ such that $\omega_0=j\Upsilon^{*}\omega_{FS}$. As $a_1^{(g_0)}=\frac{1}{2}d(d+1)$, consequently $j=1$ and $\mathrm{Ric}_{g_0}=(d+1)\omega_0$.

Let $\phi$ be the K\"{a}hler potential of the K\"{a}hler form $\omega_0$ on the domain $\Omega\subset M_0$ such that the Hermitian metric $h_0$ on $E|_{\Omega}$ can be written as
\begin{equation*}
  h_0(u,u)=e^{-\phi(z)}\|w\|^2, \;u=w\tau(z)\in E,\;z\in \Omega.
\end{equation*}
Thus
\begin{equation*}
  \omega_0(z)=\frac{\sqrt{-1}}{2\pi}\partial\overline{\partial}\phi(z),\;z\in \Omega
\end{equation*}
and
\begin{equation*}
  \omega(u)=\frac{\sqrt{-1}}{2\pi}\partial\overline{\partial}\left(\phi(z)+\log(1+ce^{-\phi(z)}\|w\|^2)\right),\;u=w\tau(z)\in E,\;z\in \Omega.
\end{equation*}
So
\begin{eqnarray*}
 \mathrm{Ric}_{g}(u)   &=& -\frac{\sqrt{-1}}{2\pi}\partial\overline{\partial}\log\left(c^{r}e^{-r\phi}\det(\partial\overline{\partial}\phi)(1+ce^{-\phi(z)}\|w\|^2)^{-d-r-1}\right) \\
   &=&\frac{\sqrt{-1}}{2\pi}\partial\overline{\partial}(r\phi-\log\det(\partial\overline{\partial}\phi)+(n+1)\log(1+ce^{-\phi(z)}\|w\|^2))
\end{eqnarray*}
for $u=w\tau(z)\in E$ and $z\in \Omega$.

By $\mathrm{Ric}_{g_0}|_{\Omega}=-\frac{\sqrt{-1}}{2\pi}\partial\overline{\partial}\log\det(\partial\overline{\partial}\phi)$ and $\mathrm{Ric}_{g_0}=(d+1)\omega_0$,
 we get $\mathrm{Ric}_{g}=(n+1)\omega$ on $E$. Using $\omega$ can be extended across $M-E$, then $\mathrm{Ric}_{g}=(n+1)\omega$ on $M$. It follows from \cite{Kobayashi-Ochiai}
that $M$ is biholomorphically isomorphism to the complex projective space $\mathbb{CP}^n$.

\end{proof}

\section{Proof of Corollary \ref{Cor:1.4}  }

\begin{proof}[\textbf{Proof of Corollary} \ref{Cor:1.4}]

It is well-known that $\mathbb{CP}^1$ is defined by
\begin{equation*}
  \mathbb{CP}^1:=(\mathbb{C}^2-\{0\})/\mathbb{C}^{*}=\left\{[z_0,z_1]:|z_0|+|z_1|\neq 0\right\},
\end{equation*}
where $\mathbb{C}^{*}$ acts by multiplication on $\mathbb{C}^2$.

The tautological line bundle $\mathcal{O}_{\mathbb{CP}^1}(-1)$ over $\mathbb{CP}^1$ is defined by
\begin{equation*}
 \mathcal{O}_{\mathbb{CP}^1}(-1):=\{([z_0,z_1],w)\in \mathbb{CP}^1\times\mathbb{C}^2:w=\lambda(z_0,z_1),\lambda\in\mathbb{C}\},
\end{equation*}
its dual line bundle, denoted by $\mathcal{O}_{\mathbb{CP}^1}(1)$. Set
\begin{equation*}
  \mathcal{O}_{\mathbb{CP}^1}(-k):=\mathcal{O}_{\mathbb{CP}^1}(-1)^{\otimes k},\;\mathcal{O}_{\mathbb{CP}^1}(k):=\mathcal{O}_{\mathbb{CP}^1}(1)^{\otimes k}
\end{equation*}
for a positive integer $k$.

Let $\mathbb{CP}^1 =U_0 \cup U_1$ be the standard open covering, where
\begin{equation*}
  U_0:=\left\{[z_0,z_1]\in  \mathbb{CP}^1: z_0\neq 0\right\}=\left\{[1,w_0]\in  \mathbb{CP}^1: z\in\mathbb{C}\right\}
\end{equation*}
and
\begin{equation*}
  U_1:=\left\{[z_0,z_1]\in  \mathbb{CP}^1: z_1\neq 0\right\}=\left\{[w_1,1]\in  \mathbb{CP}^1: w\in\mathbb{C}\right\},
\end{equation*}
then $U_i$ biholomorphically equivalent to the domain $\mathbb{C}$.

For the holomorphic line bundle $\mathcal{O}_{\mathbb{CP}^1}(k)$ over $\mathbb{CP}^1$, there exists a trivialization

\begin{equation*}
 \theta_{\alpha}:\left\{\begin{array}{rcc}
                   \mathcal{O}_{\mathbb{CP}^1}(k)\supset \pi_0^{-1}(U_{\alpha}) & \longrightarrow & U_{\alpha}\times \mathbb{C}, \\\\
                   \zeta_{\alpha}\tau_{\alpha}(w_{\alpha}) & \longmapsto & (w_{\alpha},\zeta_{\alpha})
                 \end{array}\right.
\end{equation*}
for $\alpha=0\;\text{or}\; 1$, where
\begin{equation*}
\tau_{\alpha}(w_{\alpha})=\theta^{-1}_{\alpha}(w_{\alpha},1),\;w_0=\frac{1}{w_1},\;\zeta_0=\frac{\zeta_1}{w_1^k}.
\end{equation*}

A smooth Hermitian metric $h_{FS}^k$ of $\mathcal{O}_{\mathbb{CP}^1}(k)$ is determined by
\begin{equation*}
  h_{FS}^k(\mathbf{v},\mathbf{v}):=\frac{|\zeta_{\alpha}|^2}{(1+|w_{\alpha}|^2)^k},\;\mathbf{v}=\zeta_{\alpha}\tau_{\alpha}(w_{\alpha})\in \pi^{-1}(U_{\alpha})\subset\mathcal{O}_{\mathbb{CP}^1}(k),\alpha=0,1,
\end{equation*}
its the Chern curvature determined by
\begin{equation*}
  \omega:=-\frac{\sqrt{-1}}{2\pi}\partial\overline{\partial}\log\frac{1}{(1+|w_{\alpha}|^2)^k}=\frac{k\sqrt{-1}}{2\pi}\partial\overline{\partial}\log(1+|w_{\alpha}|^2),\;\alpha=0,1.
\end{equation*}

Let $M_0:=\mathbb{CP}^1$, $L_0:=\mathcal{O}_{\mathbb{CP}^1}(k)$,  $h_0:=h_{FS}^k$ and $\omega_0:=\omega$. For $m\geq 1$,  by $L_0^m=\mathcal{O}_{\mathbb{CP}^1}(mk)$ and the spaces $H^{0}(M_0,L_0^m)$ of holomorphic sections of $L_0^m$ are equal to the spaces of homogeneous polynomials of degrees $mk$ on $\mathbb{C}^2$, we have that there exist  isometric isomorphisms between
\begin{equation*}
  \mathcal{H}_m^2(M_0)=\left\{s\in H^{0}(M_0,L_0^m):\|s\|^2:=\int_{M_0}h^m_0(s,s)\omega_{0}<+\infty\right\}
\end{equation*}
and
\begin{equation*}
  \mathcal{H}_{m}^2(\mathbb{C}):=\left\{ f\in \mbox{\rm Hol}\big({\mathbb{C}}\big): \|f\|^2:=\frac{k\sqrt{-1}}{2\pi} \int_{\mathbb{C}}|f(z)|^2
\frac{1}{(1+|z|^2)^{mk+2}}dz\wedge d\bar{z}<+\infty\right\}.
\end{equation*}

Now that the reproduce kernels of $ \mathcal{H}_{m}^2(\mathbb{C})$ are
\begin{equation*}
  K_m(z,\bar{z})=\left(m+\frac{1}{k}\right)(1+|z|^2)^{mk},
\end{equation*}
it follows that the Bergman functions for $(L_0^m,M_0,h_0^m)$ are $\epsilon_{m g_0}=m+\frac{1}{k}$ on $U_0$ or $U_1$, so $\epsilon_{m g_0}=m+\frac{1}{k}$ on $M_0$.

 The above shows that $(L_0,M_0,h_0)=(\mathcal{O}_{\mathbb{CP}^1}(k),\mathbb{CP}^1,h_{FS}^k)$ satisfies the conditions of Theorem \ref{Th:1.3}, according to Theorem \ref{Th:1.3}, Corollary \ref{Cor:1.4} holds.

\end{proof}

\vskip 20pt

\noindent\textbf{Acknowledgments}\quad The author would like to thank the referee for many helpful suggestions. The author was supported by the Scientific
Research Fund of Leshan Normal University (No. ZZ201818).

\addcontentsline{toc}{section}{References}
\phantomsection
\renewcommand\refname{References}
\small{
}

\clearpage
\end{document}